\documentclass[11pt,a4paper,final,oneside,reqno]{amsart} 
\usepackage{latexsym, amsfonts, amsthm, amsmath, amssymb, tikz, enumitem, booktabs, float,color,multicol,epsf, url,epsfig,epstopdf, array, setspace, graphicx, xcolor,tikz-cd}

\usepackage[
  a4paper,
  margin=3cm
]{geometry}
\usepackage{amsfonts}
\usepackage{amsmath}
\usepackage{amssymb}
\usepackage[applemac]{inputenc}
\usepackage{url}
\usepackage{hyperref}
\usepackage{color}
\usepackage{lipsum}
\usepackage[title]{appendix}
\usepackage{mathtools}

\usepackage{graphicx}

\def\R{\mathbb{R}}

\def\Z{\mathbb{Z}}
\def\mS{\mathbb{S}}
\def\N{\mathbb{N}}

\def\M{\mathcal{M}}
\def\F{\mathcal{F}}

\def\Isom{{\operatorname{Isom}}}

\def\Aff{\operatorname{Aff}}
\def\Sim{\operatorname{Sim}}

\def\Hol{{\operatorname{Hol}}}

\def\Euc{{\operatorname{Euc}}}

\def\Heis{{\operatorname{Heis}}}

\def\O{{\operatorname{O}}}
\def\SO{{\operatorname{SO}}}

\def\SL{{\operatorname{SL}}}

\def\Stab{{\operatorname{Stab_1}} }
\def\Stab{{\operatorname{Stab}} }

\def\dS{{\operatorname{dS}} }

\def\AdS{{\operatorname{AdS}} }

\def\Id{{\operatorname{Id}} }

\def\det{{\operatorname{det}}}

\newcommand{\Ad}{\mathrm{Ad}}
\def\Heis{{\operatorname{Heis}}}

\def\Sol{{\operatorname{Sol}}}

\def\M{{\widetilde{M}}}

\def\CW{\operatorname{CW}}

\newtheorem{theorem}{{Theorem}}[section]
\newtheorem{pr}[theorem]{{Proposition}}
\newtheorem{isom.ext}[theorem]{{Trivial isometric extension}}
\newtheorem{lemma}[theorem]{{Lemma}}
\newtheorem{cor}[theorem]{{Corollary}}
\newtheorem{fact}[theorem]{{Fact}}

\theoremstyle{definition}
\newtheorem{defi}[theorem]{{Definition}}
\newtheorem{remark}[theorem]{{Remark}}
\newtheorem{example}[theorem]{{Example}}
\newtheorem{claim}[theorem]{{\sc Claim}}

\newtheorem{notation}[theorem]{Notation}

\definecolor{purple}{rgb}{0.65,0.12,0.94}
\definecolor{forestgreen}{rgb}{0.4,0.64,0.13}
\usepackage{xcolor}
\hypersetup{
	colorlinks=true,
	linkcolor=blue,
	filecolor=blue,      
	urlcolor=blue,
    citecolor={violet}
}

\title[]{Completeness of compact Locally symmetric Lorentz manifolds}

\author [S. Allout]{Souheib Allout}
\address{Fakult\"at f\"ur Mathematik,  Universit\"at zu K\"oln,  \hfill\hfill\break\indent K\"oln, Germany}
\email{allout@math.uni-koeln.de}

\author [M. Hanounah]{Malek Hanounah}
\address{Institut f\"ur Mathematik und Informatik, Greifswald Universit\"at, \hfill\hfill\break\indent Greifswald, Germany}
\email{malek.hanounah@uni-greifswald.de}
\begin{document}

\maketitle	

\begin{center}
\begin{abstract}
The geodesic completeness of compact locally symmetric Lorentz manifolds has been established in several important cases, namely, the constant curvature, indecomposable, and Brinkmann settings. In this paper, we prove geodesic completeness in all remaining cases, thereby confirming the conjecture that all compact, locally symmetric Lorentz manifolds are geodesically complete. Along the way, we use the completeness result  to give a comprehensive overview on four-dimensional compact locally symmetric Lorentz manifolds.
\medskip

\end{abstract}
\end{center}
\tableofcontents

\section{Introduction}

A compact homogeneous semi-Riemannian manifold is geodesically complete \cite{marsden1973completeness}. However, local homogeneity alone does not guarantee completeness. Concrete examples are the projections of left-invariant metrics on compact quotients $\Gamma\backslash\SL(2,\R)$ where the Lorentz incomplete ones constitute an open cone in the space of such metrics \cite{GuediriLafontaine}. Beyond only the homogeneity condition, there is a locally homogeneous ``almost symmetric'' compact Lorentz manifold which is geodesically incomplete! Such an  example is given by a compact quotient of an incomplete homogeneous \textit{plane wave}, see \cite[Example 4.4]{hanounah2025completeness}. The almost symmetric property is seen on the level of the curvature tensor, namely, for a plane wave one has $\nabla_XR=0$ for all $X$ tangent to a codimension-one totally geodesic flat lightlike distribution ($R$ denotes the Riemann curvature tensor). So in this sense, one can see the completeness of compact locally symmetric Lorentz manifolds as a ``sharp'' result.

One of the main contributions to the completeness problem of locally symmetric spaces, was by Carri\`ere \cite{carriere1989autour}, showing that compact flat Lorentz manifolds are geodesically complete. Later on, Klingler \cite{klingler1996completude} generalized this result to the constant curvature case.

On the other hand, non-constant sectional curvature indecomposable Lorentz symmetric spaces are known as \textit{Cahen--Wallach spaces}. They have a distinguished geometry that differs from that of constant curvature. In particular, they admit a globally defined (unique up to scale) lightlike parallel vector field. Compact manifolds locally isometric to a Cahen--Wallach space are shown to be complete relatively recently by  Leistner and Schliebner \cite{leistner2016completeness}. In fact, they showed the geodesic completeness for the larger class of compact \textit{pp-waves} (i.e. compact Lorentz manifolds that admit a parallel lightlike vector field $V$ whose orthogonal distribution $V^\perp$ is \textit{flat}). The class of pp-waves fits in the larger class of \textit{Brinkmann spacetimes}, which are defined by the existence of a parallel lightlike vector field, dropping the flatness assumption on the orthogonal distribution. Recently,  Mehidi and Zeghib \cite{mehidi2026completeness} extended the completeness results to this larger class of spacetimes.
 
Despite all the aforementioned completeness results, the general case of locally symmetric Lorentz manifolds remained open. Very recently, Leistner and Munn \cite{Leistnercompletness} observed that the completeness result in \cite{mehidi2026completeness} can be applied to establish the completeness of compact locally symmetric Lorentz manifolds \( M \), whose indecomposable Lorentz factor, in the (local) de Rham--Wu decomposition, is either a Cahen--Wallach space or whose maximal flat factor is isometric to \( (\mathbb{R}, -dt^2) \) (see Sect. \ref{Section: symmetric} for the precise definition).\\

Our main result is the following theorem.
 
\begin{theorem}\label{Intro: the completeness thm}
    A compact locally symmetric Lorentz manifold is geodesically complete.
\end{theorem}

\subsection{Calabi--Markus phenomenon} The celebrated Calabi--Markus phenomenon states that de Sitter spacetimes do not admit compact Clifford--Klein space forms. The completeness result of Klingler \cite{klingler1996completude} implies then the non-existence of compact Lorentz manifolds with positive constant sectional curvature. This naturally raises the question of whether this phenomenon extends to the product setting. Specifically, one may inquire whether there exist compact locally symmetric Lorentz manifolds for which the de Sitter space is the indecomposable Lorentz factor.
\begin{pr}\label{Calabi--Markus thm} Let $M$ be a compact locally symmetric Lorentz manifold. Then the indecomposable Lorentz factor, in the de Rham--Wu decomposition of its symmetric model, is not a de Sitter space $\dS^{1,n}$, $n\geq 2$.    
\end{pr}

\subsection{Beyond the Lorentz signature}
The completeness problem for compact locally symmetric manifolds remains largely open in the higher signature setting. Even the flat case of signature $(2,2)$ is completely unknown, i.e. it is still an open problem whether a compact flat semi-Riemannian $4$-manifold of signature $(2,2)$ is geodesically complete. This is a very particular case of Markus' conjecture, which asserts that a compact \textit{affine manifold} is complete if and only if it is \textit{unimodular}, i.e. it has a parallel volume form (see for example \cite[Chapter 11]{goldman2022geometric}).

\subsection*{Acknowledgment} The second author is grateful to his PhD adviser Ines Kath for her valuable feedback on an early draft of this paper. The first author would like to thank his PhD adviser Stefan Suhr for his support and helpful comments on this manuscript. 

The first author is supported by the SFB/TRR 191 `Symplectic Structures in Geometry, Algebra and Dynamics', funded by the DFG (Projektnummer 281071066-TRR191).

\section{Symmetric spaces}\label{Section: symmetric}

\begin{defi}\label{Defi: indecomposable}
    Let $X$ be a simply connected pseudo-Riemannian symmetric space and let $p\in X$. We say that $X$ is indecomposable if there is no proper non-degenerate subspace in $T_pX$ invariant by the holonomy group $\Hol_p$ at the point $p$.
\end{defi}
A simply connected Lorentz symmetric space $X$ can be decomposed isometrically (see \cite{deRham, Wu}) as a product $X=X^0\times X^1\times \cdots \times X^n$ where $X^0$ is a \textit{maximal flat factor} of $X$ (it is a maximal non-degenerate factor on which the holonomy acts trivially) and each $X^i$ is an indecomposable pseudo-Riemannian symmetric factor. Moreover, this decomposition is unique up to isometry and permutation, and it is called \textit{the de Rham--Wu decomposition}.  We say that $X$ has a \textit{flat Lorentz factor} if $X^0$ is Lorentz (this includes, by convention, the case $X^0\cong (\R,-dt^2)$).
\begin{notation}
    We denote by $X^L$ the Lorentz factor of $X$ whether it is a maximal flat factor or indecomposable. 
\end{notation}
    
\begin{fact}[\cite{cahenwallach}]\label{Fact: indecomposable}
Let $X$ be a simply connected indecomposable Lorentz symmetric space. Then either $X$ has non-zero constant sectional curvature, or isometric to a Cahen--Wallach space, or isometric to the line $(\R,-dt^2)$.
\end{fact}

\begin{remark}[Cahen--Wallach spaces]\label{Remark: CW}
    The geometry of Cahen--Wallach spaces is quite different compared to the constant curvature case. These spaces have a (unique) parallel lightlike vector field $V$ which is invariant by a subgroup of index two in the isometry group. Moreover, their full isometry groups are amenable, i.e. $K\ltimes R$ where $K$ is compact and $R$ is solvable. For more details, see \cite[Section 2]{KO}. 
\end{remark}

\subsection{Splitting of the isometry group} In this section we show that the isometry group of a decomposable Lorentz symmetric space also decomposes in a natural way if the Lorentz factor is not a Cahen--Wallach space.
\begin{pr}\label{splitting isometry}
     Let $X=X^L\times Y$ be a simply connected Lorentz symmetric space. Assume that $X^L$ has constant sectional curvature. Then $\Isom(X)=\Isom(X^L)\times \Isom(Y)$. 
\end{pr}

\begin{proof} 
Let us first treat the case where $X^L$ is the de Sitter or anti de Sitter space. For $p=(p_1,p_2)\in X$ we have $\Hol_p=\Hol_{p_1}\times \Hol_{p_2}$. The stabilizer $\Stab(p)\subset \Isom(X)$ normalizes $\Hol_{p_1}\times \Hol_{p_2}$. But $\Hol_{p_1}=\SO^o(1,n)$ and $\Hol_{p_2}$ is compact. This implies that $\Stab(p)$ normalizes $\SO^o(1,n)$. So, $\Stab(p)$ leaves invariant the fixed space of $\SO^o(1,n)$ which is nothing but the subspace tangent to the $Y$-direction at $p$. Hence, it also leaves invariant its orthogonal i.e. the subspace tangent to the $X^L$-direction at $p$. This means that $\Isom(X)$ preserves the splitting $X=X^L\times Y$.

Assume now that $X^L$ is the maximal flat Lorentz factor, so it is isometric to a Minkowski space. Then $Y=G/K$ is a symmetric Riemannian space with a trivial maximal flat factor, in particular with semisimple isometry group $G$. The holonomy group at a point $p$ coincides with the isotropy $K$ of the factor $Y$ (see \cite[Proposition 10.79]{besse2007einstein}). On the other hand, the fixed space of the $K$-action at $p$ is exactly $X^L$. Therefore, the stabilizer of $p$ preserves the factor $X^L$. The rest follows similarly as above.
\end{proof}

\section{Injectivity of the developing map}
In this section we assume that $X=X^L\times Y$ is a simply connected Lorentz symmetric space whose Lorentz factor $X^L=X_{\kappa}$ is the universal model of constant curvature $\kappa$. Let $M$ be a compact $(\Isom(X),X)$-manifold and we fix a developing pair $(D,\rho)$ of $M$.
\subsection{Natural foliations} Proposition \ref{splitting isometry} implies that a compact manifold $M$ locally isometric to $X=X^L\times Y$ inherits two transverse foliations $\F_1$ and $\F_2$ where the leaves of $\F_1$ are Lorentz of constant curvature locally modeled on $X_{\kappa}$ and the leaves of $\F_2$ are Riemannian modeled on $Y$. The same applies to $\widetilde{M}$, we have two foliations $\widetilde{\F}_1$ and $\widetilde{\F}_2$.

\begin{lemma}\label{compleness riemannian lemma}
    Let $M$ be a compact manifold locally isometric to $X$. Each leaf of $\widetilde{\F}_2$ is mapped, under the developing map, bijectively and isometrically onto a vertical Riemannian fiber $\left\{x_0\right\}\times Y$.
\end{lemma}
\begin{proof}
    The compactness of the manifold $M$ implies that each leaf of the foliation $\F_2$ is complete (with respect to its Riemannian structure locally modeled on $Y$). Hence each leaf $\widetilde{\F}_2(p)$ of the foliation $\widetilde{\F}_2$ is also complete. The developing map $D$, restricted to $\widetilde{\F}_2(p)$, is a local isometry between $\widetilde{\F}_2(p)$ and a corresponding fiber $\left\{x_0\right\}\times Y$. Since $\widetilde{\F}_2(p)$ is complete, then $D:\widetilde{\F}_2(p)\to \left\{x_0\right\}\times Y$ is a covering map. But $Y$ being simply connected implies that $D:\widetilde{\F}_2(p)\to \left\{x_0\right\}\times Y$ is in fact an isometric diffeomorphism.
\end{proof}

\begin{cor}\label{image universal cover}
    The universal cover $\widetilde{M}$ is mapped, under the the developing map, onto $\Omega\times Y$ where $\Omega\subset X_{\kappa}$ is an open subset.
    
\end{cor}

\subsection{A natural action on $\widetilde{M}$} Since each leaf of $\widetilde{\F}_2$ is identified canonically under the developing map to $Y=G/K$, we obtain a well-defined $G$-action on $\widetilde{M}$ for which the developing map is $G$-equivariant. Since the developing map is isometric, we have
\begin{cor}
    The $G$-action on $\widetilde{M}$ is isometric, proper, and with orbits the $\widetilde{\F}_2$-leaves.
\end{cor}

\subsection{Product structure of the universal cover} We have seen in Corollary \ref{image universal cover} that $D(\widetilde{M})=\Omega\times Y$ where $\Omega\subset X_{\kappa}$ is an open subset. The subset $\widehat{\Omega}=D^{-1}\left(\Omega\times\left\{y\right\}\right)$ is a global cross section of the $\widetilde{\F}_2$-foliation. Indeed, consider the map $\sigma:\widetilde{M}\to Y$ where $\sigma=\pi\circ D$ and $\pi:X_{\kappa}\times Y\to Y$ is the natural projection. We have $\sigma^{-1}(y)=\widehat{\Omega}$ and, by construction, $\sigma$ is $G$-equivariant. Since each leaf $\widetilde{\F}_2(p)$ is identified under the map $\sigma$ with $Y$, then $\widetilde{\F}_2(p)$ intersects $\sigma^{-1}(y)=\widehat{\Omega}$ exactly in one point. In other words, $\widehat{\Omega}$ is a leaf of $\widetilde{\F}_1$ which is a global cross section of the foliation $\widetilde{\F}_2$. 

\begin{cor}\label{direct product}
    The universal cover $\widetilde{M}$ is globally isometric to $\widehat{\Omega}\times Y$.
\end{cor}

In particular, $\widehat{\Omega}$ is a connected and simply connected Lorentz manifold of constant curvature $\kappa$. The injectivity of the developing map $D:\widetilde{M}\to X_{\kappa}\times Y$ is then equivalent to the fact that $\widehat{\Omega}$ isometrically embeds in $X_{\kappa}$.

\subsection{The injectivity} We have seen in Corollary \ref{direct product} that $\widetilde{M}=\widehat{\Omega}\times Y$ and $\pi_1(M)=\Gamma\subset \Isom(\widehat{\Omega})\times G$ acts freely, properly, and cocompactly on $\M$ where $G$ is the isometry group of $Y$. Frances in \cite[Sect. 6]{francescoarse} considers, among many other things, a similar situation (in fact a more general setting of warped products). He observed \cite[Proposition 6.9]{francescoarse} that the injectivity result in the works of Carri\`ere and Klingler \cite{carriere1989autour, klingler1996completude} can be adapted to the product setting. That is,
\begin{pr}[\cite{francescoarse}, Sect. 6.3]\label{injectivity frances} $\widehat{\Omega}$ is isometric to $X_{\kappa}$ or to a convex open subset $\Omega$ of $X_{\kappa}$ whose boundary is either a lightlike hyperplane or the disjoint union of two lightlike hyperplanes.
    
\end{pr}
As mentioned above, the ideas in \cite{carriere1989autour, klingler1996completude} can be carried out in our product situation. Indeed, let $\tilde{p}\in \M$ and let $\gamma$ be an incomplete geodesic starting at $\tilde{p}$ and contained in $\widetilde{\F}_1(\tilde{p})$ the Lorentz leaf through $\tilde{p}$. Consider the visibility set $E_{\tilde{p}} \subset  \widetilde{\F}_1(\tilde{p})$ (see \cite[Section 2]{klingler1996completude} for the definition). Let $\{\tilde{v}_n\}_{n\in \N}\subset \gamma$ be a sequence that goes outside all compact subsets. Since $M$ is compact then, up to multiplying $\{\tilde{v}_n\}_{n\in \N}$ by a well-chosen sequence $\alpha_n\in \pi_1(M)$, we can assume that $\tilde{w}_n\mathrel{\mathop:}=\alpha_n\tilde{v}_n$ is convergent.  The limit point $\tilde{w}$, of the sequence $\{\tilde{w}_n\}_{n\in \N}$, belongs a priori to a different $\widetilde{\F}_1$-leaf. Choose a small enough product ball $\widetilde{B}=\widetilde{B}_1\times \widetilde{B}_2$ around $\tilde{w}$, so that $\alpha \widetilde{B}\cap \widetilde{B}=\varnothing$ for all $\alpha \in \pi_1(M)$ and $\widetilde{B}$ develops injectively. We obtain a sequence $\{C_n\}_{n\in \N}$ of disjoint ellipsoids around $D(\tilde{v}_n)$ where  $C_n\mathrel{\mathop:}= D(\alpha_n^{-1}\widetilde{B}_1)=\rho(\alpha_n^{-1})B_1$. Moreover, up to multiplication on the left and the right by two sequences living in a compact neighborhood of the identity, the sequence $g_n\mathrel{\mathop:}=\rho(\alpha_n^{-1})$ can be assumed to be in the stabilizer of the point $D(\tilde{w})$. Since $\Stab(D(\tilde{w}))$, restricted to the Lorentz leaf passing through $D(\tilde{w})$, has discompacity $1$, we conclude that the sequence of ellipsoids  $\{C_n\}_{n\in \N}$ converges to a degenerate ellipsoid of codimension $1$  (see \cite{carriere1989autour, klingler1996completude}). Hence the Carri\`ere--Klingler proof of the convexity (of the visibility set). Having the convexity at hand, the description of the boundary follows similarly to \cite[Proposition 3]{klingler1996completude}.

\begin{cor}\label{boundaries}
    We have that $M=\Gamma\backslash\left(\Omega\times Y\right)$ where $\Gamma\subset \Isom(X_{\kappa})\times G$ is a discrete subgroup acting freely, properly, and cocompactly on $\Omega\times Y$ and $\Omega\subset X_{\kappa}$ is a convex open subset as concluded in Proposition \ref{injectivity frances}.
\end{cor}

\section{The case of $\AdS$ or $\dS$ Lorentz factor} In this section, $X_{\kappa}$ denotes the Lorentz universal model of constant sectional curvature $\kappa=\pm 1$. That is, $X_{\kappa}$ either $\widetilde{\AdS}^{1,n+1}$ or $\dS^{1,m+1}$ for $m\geq 1$.
\begin{pr}\label{complete ads} Let $M$ be a compact Lorentz manifold locally isometric to $X_{\kappa}\times Y$ where $Y$ is a homogeneous Riemannian manifold. Then $M$ is complete.
\end{pr}
\begin{proof}
    As in Proposition \ref{splitting isometry} we have $\Isom(\Omega\times Y)=\Isom(\Omega)\times\Isom(Y)$. Applying \cite[Subsect. 6.3.1]{francescoarse} we get our desired conclusion.
\end{proof}

The ideas in \cite[Subsect. 6.3.1]{francescoarse} can be briefly summarized as follows. Suppose that the boundary of \(\Omega\) consists of a single lightlike hyperplane (this is the case for $\kappa=1$). Then \(M = \Gamma \backslash (\Omega \times Y)\) admits a complete vector field (possibly defined on only a strict open subset) with non-zero constant divergence, which contradicts the compactness of \(M\).\\
If the boundary of \(\Omega\) is the disjoint union of two lightlike hyperplanes (so \(\kappa = -1\)), then either they are parallel, in which case \(M\) again admits a complete vector field with nonzero constant divergence, or they are not parallel, in which case there exists a well-defined unbounded function on \(M\). In both cases, this contradicts the compactness of \(M\).\\

It is worth noting that, in the absence of a global invariant volume, we have the following counterexample
\begin{example}[Transversally affine $\AdS$-foliation]
    The group $\widetilde{\SO}(1,2)\times \Aff(\R)$ acts transitively on $\widetilde{\AdS}^{1,1}\times \R$. The subgroup $$\Aff^+(\R)\times \Aff^+(\R)\subset \widetilde{\SO}(1,2)\times \Aff(\R)$$
    has an open orbit $\mathcal{O}\times\R\subset \widetilde{\AdS}^{1,1}\times \R$ where $\Aff^+(\R)\subset \widetilde{\SO}(1,2)$ acts freely transitively on $\mathcal{O}$ (actually $\partial \mathcal{O}$ consists of two parallel lightlike geodesics). There is an embedding of $\Sol=\SO^o(1,1)\ltimes \R^2$ into $\Aff^+(\R)\times \Aff^+(\R)$ with a free transitive action on $\mathcal{O}\times\R$. Taking a cocompact lattice $\Gamma\subset \Sol$, we obtain a compact quotient $\Gamma\backslash(\mathcal{O}\times\R)$ with incomplete anti de-Sitter leaves (in fact the leaves are transversally affine in the sense of \cite[p. 36]{carriere1984flots} or \cite{blumenthal}).
\end{example}

\subsection{The de Sitter case: Calabi--Markus phenomenon} 
\begin{pr}\label{general foliated Calabi--Markus} Let $M$ be a compact locally symmetric Lorentz manifold. Then the (local) indecomposable Lorentz factor of $M$ is not a de Sitter space $\dS^{1,n}$, $n\geq 2$.
\end{pr}
\begin{proof}
    We argue as in the original proof of \cite{calabimarkus}. Recall that any two spacelike totally geodesic spheres (of codimension $1$) in $\dS^{1,n}$ intersect non-trivially. Assume there is $\Gamma\subset \SO(1,n+1)\times G$ acting properly discontinuously and cocompactly on $\dS^{1,n}\times G$ where $G$ is any connected Lie group. If the projection of $\Gamma$ to $G$, denoted $\widehat{\Gamma}$, is discrete then $\Gamma\cap \SO(1,n+1)$ acts properly cocompactly on $\dS^{1,n}$ (see Lemma \ref{useful lemma}) which is impossible by \cite{calabimarkus}. Hence $\widehat{\Gamma}$ is non-discrete which implies that there is a compact $K\subset G$ and an infinite sequence $\widehat{\gamma}_n$ of distinct elements of $\widehat{\Gamma}$ such that $\widehat{\gamma}_nK\cap K$ is non-empty for all $n$. Let $\gamma_n$ be a lift of $\widehat{\gamma}_n$ to $\Gamma$ and $S\subset \dS^{1,n}$ be a spacelike totally geodesic hypersphere. It follows that $\gamma_n(S\times K)\cap (S\times K)$ is non-empty for all $n$ which contradicts the fact that $\Gamma$ acts properly on $\dS^{1,n}\times G$. To conclude the proof of Proposition \ref{general foliated Calabi--Markus}, observe that if $\Gamma\subset \SO(1,n+1)\times G$ acts properly discontinuously and cocompactly on a symmetric space $\dS^{1,n}\times Y$ then so it does on $\dS^{1,n}\times G$ which we showed to be impossible.
\end{proof}

\subsection{The foliated setting} Theorem \ref{complete ads} shows that a compact manifold $M$ locally modeled on the geometry $\bigl(\Isom(X_{\kappa})\times\Isom(Y),X_{\kappa}\times Y \bigr)$ is complete where $\kappa=\pm 1$ and $Y$ is any homogeneous Riemannian manifold. In particular, the leaves of the $X_{\kappa}$-foliation on $M$ are (immersed) complete Lorentz manifolds of constant curvature $\kappa$.\\

One naturally asks whether a merely foliated version still holds. More precisely, suppose that $M$ is a compact manifold endowed with an $X_{\kappa}$-foliation $\F$, i.e. each leaf of $\F$ is endowed with an $X_{\kappa}$-structure in a continuous way (see for example \cite[Def. 2.3]{hanounah2025completeness} or \cite{inaba1993tangentially}). Is it true that the leaves of $\F$ are complete? If one does not impose any ``transverse'' property on the foliation $\F$, then the answer to the question is negative as Example \ref{incomplete dS} shows.

On the other hand, the Calabi--Markus phenomenon still holds in a merely foliated setting (even for laminations!). More precisely, a compact manifold $M$ does not admit a foliation by \textit{complete} $\dS^{1,n}$-leaves. Let us give a sketch of the idea of the proof. Suppose by contradiction that $M$ admits such a foliation $\F$. By the standard Calabi--Markus phenomenon, we know that each leaf of $\F$ is not compact. Denote by $\overline{\F(p)}$ the topological closure of a leaf $\F(p)$ and let $y\in \overline{\F(p)}$ be a point that does not belong to $\F(p)$. We denote by $\operatorname{Gr_n^{+}}(\F)$ the Grassmannian of spacelike $n$-planes tangent to the foliation $\F$. One can find a sequence $P_{y_i}\in \operatorname{Gr_n^{+}}(\F)$ that converges to $P_y\in \operatorname{Gr_n^{+}}(\F)$ where $y_i\in \F(p)$ converges to $y$. Each $P_{y_i}$ is tangent to a unique totally geodesic spacelike hypersphere $S_{i}\subset\F(p)$. The limit hypersphere $S$, tangent to $P_y$, lives in a different leaf $\F(y)$. Hence one can find a neighborhood $N(S)$ of $S$ that separates $S_1$ from $S$. But eventually, $S_i$ is contained in $N(S)$. This is impossible since each $S_i$ intersects $S_1$ which shows that such a foliation cannot exist. The same ideas generalize to laminations, i.e. there cannot be a lamination by complete de Sitter leaves.

The completeness assumption of the leaves is crucial as the following example shows.

\begin{example}[An incomplete de Sitter foliation]\label{incomplete dS} Let $\SO^o(1,2)\subset\SO^o(1,3)$ be the isotropy subgroup of a point in $\dS^{1,2}$ and $P=\R\ltimes \R^2\subset\SO^o(1,3)$ be a Borel subgroup (i.e., $P=AN$ in the $KAN$ decomposition of $\SO^o(1,3)$). The subgroup $\SO^o(1,2)$ can be conjugated inside $\SO^o(1,3)$ so that it becomes transverse to $P$. In other words, $P$ has an open orbit when acting on $\dS^{1,2}$. This open orbit is necessarily strict, i.e. $P$ does not act transitively on $\dS^{1,2}$ (equivalently, $P$ has a left invariant incomplete Lorentz metric of positive constant curvature). Consider the left action of $P$ on $\SO^o(1,3)/\Gamma$ where $\Gamma\subset\SO^o(1,3)$ is a cocompact lattice. The orbits of this action define a foliation by incomplete $\dS^{1,2}$-leaves.
    
\end{example}

\section{The case of flat Lorentz factor} It remains to treat the case where the maximal flat factor of $X$ is Lorentz. Namely, $X=\mathbb{R}^{1,n+1}\times Y$ where $Y=G/K$ is a Riemannian symmetric space with a (algebraic) semisimple isometry group $G$ and $K\subset G$ is a maximal compact. In the particular situation where the Riemannian symmetric factor $Y$ is trivial, Carri\`ere showed that the developing  map $D:\widetilde{M}\to \R^{1,n+1}$, of a compact manifold $M$ modeled on $X$ (i.e. flat), is either a diffeomorphism (i.e. $M$ is complete) or an embedding onto an open convex subset $C_n\subset \R^{1,n+1}$ which is either a half-Minkowski space whose boundary is a lightlike hyperplane or the boundary of $C_n$ consists of two parallel lightlike hyperplanes. He then eliminated the latter two cases using a result of Goldman and Hirsch \cite{Goldmanhirsch2} that asserts that the (affine) holonomy of a compact affine manifold with a parallel volume is irreducible, i.e. does not preserve a proper affine subspace. This argument does not extend to our product situation. In fact, Frances pointed out the same observation \cite[Subsect. 6.3.2]{francescoarse} in the warped product setting where he showed the completeness assuming further dynamical properties on the isometry group of the compact manifold.\\

We have seen in Corollary \ref{boundaries} that if $M$ is a Lorentz manifold locally isometric to a symmetric space $\R^{1,n+1}\times Y$ then either $\widetilde{M}$ is complete i.e. isometric to $\R^{1,n+1}\times Y$, or $\widetilde{M}$ is isometric $C_n\times Y$ where $C_n\subset \R^{1,n+1}$ is an open convex subset whose boundary is a lightlike hyperplane or the disjoint union of two parallel lightlike hyperplanes. The latter case is easier to eliminate (see the lemma below). Thus, our problem reduces to the half-Minkowski case which turned out to be more involved.
\begin{lemma}\label{Lemma: half-Minkowski image}
   $C_n$ cannot be bounded between two parallel lightlike hyperplanes.
\end{lemma} 
\begin{proof}
Assume that $C_n$ is bounded by two parallel lightlike hyperplanes $P_1$ and $P_2$. The image $\Gamma_1\mathrel{\mathop:}=p_1(\Gamma)$, under the projection $p_1:  \Isom(\R^{1,n+1})\times G\to\Isom(\R^{1,n+1})$, preserves both $P_1$ and $P_2$. The stabilizer of $P_1$ is isomorphic to $(\R\times \O(n))\ltimes \Heis_{2n+1}$ (see Sect \ref{Ln}). Since $\Gamma_1$ preserves also $P_2$ then $\Gamma_1\subset \O(n)\ltimes \Heis_{2n+1}$. But this is impossible since $\left( \O(n)\ltimes \Heis_{2n+1}\right)\times G$ does not act cocompactly on $C_n\times Y$.
\end{proof}

Therefore, if $M$ is a compact locally symmetric Lorentz manifold modeled on $\R^{1,n+1}\times Y$, then either $M$ is complete or $M$ is the quotient of $C_n\times Y$ where $C_n$ is a half-Minkowski space whose boundary is a lightlike hyperplane. 

\subsection{Algebraic formulation}\label{Ln} 

\begin{defi}
  Let $P\subset \R^{1, n+1}$ be a lightlike hyperplane and $C_n$ a half-Minkowski space with boundary $P$. The subgroup $\operatorname{LP}(n)\subset\O(1,n+1)\ltimes\R^{1,n+1}$ of affine Lorentz transformations, leaving invariant the half-space $C_n$, is called the \textit{lightlike Poincar\'e group}. 
\end{defi} 
One can write the group $\operatorname{LP}(n)$ as (see Appendix \ref{lightlike poincare} for more details) \[\operatorname{LP}(n)\cong (\R\times \O(n))\ltimes \Heis_{2n+1}\]
where the $\R$-factor acts on the Heisenberg group $\Heis_{2n+1}$ by the exponential of a certain derivation of the Heisenberg algebra, we refer to Appendix \ref{lightlike poincare} for the exact computations.

\subsubsection{Half-Minkowski as a homogeneous space} The lightlike Poincar\'e group $\operatorname{LP}(n)$ acts transitively on half-Minkowski (without boundary), where the isotropy of a point is isomorphic to $\O(n)\ltimes J_n$ and $J_n\subset \Heis_{2n+1}$ an abelian (Lagrangian) subspace of dimension $n$ which does not intersect the center of $\Heis_{2n+1}$ (see Subsect. \ref{homogeneous half Minkowski}). Hence $$C_n=\operatorname{LP}(n)/\left(\O(n)\ltimes J_n\right)\cong \operatorname{LP}(n)/\Euc_n$$ 
where $\Euc_{n}\cong \O(n)\ltimes \R^{n}$ denotes the isometry group of the $n$-dimensional Euclid space. 

\subsubsection{Compact models}\label{Subsection: CK forms} The above discussions show that if $M$ is a compact Lorentz manifold locally isometric to a symmetric space $\R^{1,n+1}\times Y$ then either $M$ is complete or $M$ is the quotient of $C_n\times Y$ by a discrete subgroup $\Gamma\subset \operatorname{LP}(n)\times G$ acting properly and cocompactly on $C_n\times Y$. One observes that if such a $\Gamma\subset \operatorname{LP}(n)\times G$ exists then its action on $C_n\times G$ is also proper and cocompact. Thus, in order to prove the completeness of $M$, it suffices to eliminates this latter case. In fact, we prove the more general statement
\begin{theorem}\label{Theorem: no compact CK forms}
    Let $G$ be a semisimple real Lie group. Then there is \textbf{no} discrete subgroup $\Gamma\subset \operatorname{LP}(n)\times G$ acting  properly and cocompactly on $C_n\times G$.
\end{theorem}

\subsection{Proof of Theorem \ref{Theorem: no compact CK forms}} Put $R_n=\R\ltimes\Heis_{2n+1}\subset\operatorname{LP}(n)$ and let $$p_2:\operatorname{LP}(n)\times G=(\O(n)\times G)\ltimes R_n \to \O(n)\times  G$$ denote the quotient projection. We define $\Gamma_2\mathrel{\mathop:}=p_2(\Gamma)$.  

\begin{fact}[Auslander \cite{auslander}]\label{fact: auslander}
    Let $\Gamma$ be a discrete subgroup in $A\ltimes B$ where $A$ and $B$ are connected Lie groups with $B$ solvable. Then the identity component of the closure of the projection of $\Gamma$ to $A$ is solvable. 
\end{fact}
This implies that the identity component of the closure of $\Gamma_2$ is solvable. The following lemma will be useful several times later.

\begin{lemma}\label{useful lemma}

Let $G$ be a connected Lie group and $I\subset G$ be a connected closed Lie subgroup. Let $\Gamma\subset G$ be a discrete subgroup acting freely, properly, and cocompactly on $G/I$. Assume there is a normal closed and connected subgroup $H\subset G$ such that the projection of $\Gamma$ to $G/H$ acts freely, properly, and cocompactly on $(G/H)/(I/(I\cap H))$. Then $\Gamma\cap H$ acts freely, properly, and cocompactly on $H/(H\cap I)$.

\end{lemma}

\begin{proof}
    Denote by $\widehat{\Gamma}$ and $\widehat{I}$ the projections of $\Gamma$ and $I$ to $\widehat{G}\mathrel{\mathop:}=G/H$. By assumption, the quotient spaces $\Gamma\backslash(G/I)$ and $\widehat{\Gamma}\backslash(\widehat{G}/\widehat{I})$ are compact manifolds. The natural projection $\pi:G\to \widehat{G}$ induces a submersion $\overline{\pi}:\Gamma\backslash(G/I)\to\widehat{\Gamma}\backslash(\widehat{G}/\widehat{I})$. Hence $\overline{\pi}$ has closed fibers given by the projections of the cosets of $H/(H\cap I)$ to $\Gamma\backslash(G/I)$. These fibers are closed submanifolds which means that $\Gamma\cap H$ acts freely, properly, and cocompactly on $H/(H\cap I)$.
\end{proof}

Now, let $H\mathrel{\mathop:}=\overline{\Gamma_2}$ be the closure of $\Gamma_2$ and $H^o$ be its connected identity component. The group $H$ is a cocompact subgroup of the algebraic group $\O(n)\times  G$. If $H^o=\{e\}$ then Lemma \ref{useful lemma} implies that $\Gamma\cap \operatorname{LP}(n)$ acts freely, properly, and cocompactly on $C_n$ which is impossible. Hence $H^o$ is a non-trivial connected solvable subgroup of $\O(n)\times  G$. Let $H^{oo}$ be the projection of $H^o$ to $G$. Similarily, if $H^{oo}$ is trivial we get that  $\operatorname{LP(n)\cap \Gamma}$ acts freely, properly, and cocompactly on $C_n$ which is impossible.

\textbf{Assume for the sake of contradiction that $H^{oo}$ is non-trivial}. Since $\O(n)$ is compact, the projection to $G$ is a proper map. In particular, $H^{oo}$ is  a non-trivial, closed solvable subgroup of $G$. \\

Denote by $N=N(H^{oo})$ the normalizer in $G$ of the identity component $H^{oo}$. Hence $N$ is an algebraic group. Since the identity component  $N^o$ is of finite index in $N$ and $N$ is cocompact in $G$, then $N^o$ is a cocompact connected closed Lie subgroup of $G$.

\begin{claim}\label{Claim: N^o is unimodular} 
    $N^o$ is a unimodular Lie group. 
\end{claim}

We recall the following fact:
\begin{pr}[Proposition 3.1 \cite{zeghib1998closed}]\label{Prop: modular distorsion}
Let $G$ be a Lie group and $H \subset G$ a closed connected subgroup, with Lie algebras
$\mathfrak g$ and $\mathfrak h$, respectively. For $g\in G$, define the modular distortion
\[
\Delta_G(g)=\det\!\bigl(\operatorname{Ad}(g)\vert_{\mathfrak g}\bigr),
\]
and, for $g$ normalizing $H$,
\[
\Delta_H(g)=\det\!\bigl(\operatorname{Ad}(g)\vert_{\mathfrak h}\bigr).
\]
Then:
\begin{enumerate}
\item If $G/H$ admits a nontrivial $G$-invariant measure, then for every $h\in H$, \[\Delta_G(h)=\Delta_H(h).\]

\item If, moreover, $G/H$ admits a finite-volume quotient $\Gamma\backslash G/H$,
where $\Gamma$ is discrete, then for every $g\in N_G(H)$,
\[
\Delta_G(g)=\Delta_H(g).
\]
\end{enumerate}
\end{pr}

\begin{proof}[Proof of Claim \ref{Claim: N^o is unimodular}]
    We assume up to taking a finite index subgroup of $\Gamma$ that $\Gamma\subset \operatorname{LP}(n)\times N^o$. Moreover, the homogeneous space $C_n\times N^o$ has a compact Clifford--Klein form. However, as $N^o$ normalizes the isotropy group $I=\O(n)\ltimes J_n$, we conclude using  in Proposition \ref{Prop: modular distorsion} point $(2)$,  that for any $n\in N^o$, $\det (\Ad(n_{|I})=\det (\Ad(n_{| \operatorname{LP}(n)\times N^o})$. However, as $N^o$ centralizes $I$ we deduce even that $\det (\Ad(n_{|I}))=1$. Hence, $\det(\Ad(n_{| \operatorname{LP}(n)\times N^o})=1$. Since  $\det(\Ad(n_{| \operatorname{LP}(n)\times N^o})=\det (\Ad(n_{| N^o} ))$, the lemma follows.
\end{proof}

\begin{cor}Let $o \in G/N^o$. The adjoint action of $N^o$ on the tangent space $T_o(G/N^o)$ preserves a finite invariant volume.\end{cor}\begin{proof}Let $n \in N^o$. Since $G$ is a semisimple Lie group, its adjoint representation is unimodular, meaning $\det(\Ad(n)|_{\mathfrak{g}}) = 1$. By the preceding Lemma, the restriction to the subgroup also satisfies $\det(\Ad(n)|_{\mathfrak{n}^o}) = 1$. Using the natural identification of the tangent space $T_o(G/N^o)$ with the quotient  $\mathfrak{g}/\mathfrak{n}^o$, the determinant splits as follows:$$\det(\Ad(n)|_{\mathfrak{g}}) = \det(\Ad(n)|_{\mathfrak{n}^o}) \cdot \det(\Ad(n)|_{\mathfrak{g}/\mathfrak{n}^o})$$Substituting the known values yields $\det(\Ad(n)|_{T_o(G/N^o)}) = 1$. This last equality is equivalent  to the existence of a left $G$-invariant volume form on the homogeneous space $G/N^o$. Finally, since $G/N^o$ is compact, this invariant volume is finite.\end{proof}
\begin{claim}
    $N^o$ is Zariski dense.
\end{claim}
\begin{proof}
    By the previous corollary we know that $G/N^o$ has a $G$-invariant finite volume. Since $N^o$ is closed we get by Borel density theorem \cite{borel1960density} that $N^o$ is Zariski dense.
\end{proof}

\paragraph{\textbf{The final contradiction}}
The above contradicts the fact that $N^o$ is the contained in the normalizer of the non-trivial connected subalgebra $H^{oo}$.
\begin{cor}
    Compact Lorentz manifolds that are  modeled on a symmetric Lorentz space with a flat Lorentzian factor of dimension $\geq 2$ are geodesically complete. 
\end{cor}

\section{The remaining cases}
As mentioned in the introduction, the geodesic completeness of compact locally symmetric Lorentz manifolds, whose indecomposable Lorentz factor $X^L$ is a Cahen--Wallach space or isometric to $(\R,-dt^2)$, is shown in \cite{Leistnercompletness} using the completeness result in \cite{mehidi2022completeness}. In this section, we give a slightly different proof.
\subsection{The case of Cahen--Wallach factor}

Let $X=\CW\times Y$ be a simply connected symmetric space where $Y$ is a symmetric Riemannian space and $\CW$ is a Cahen--Wallach space. Denote by $Y^0$ the maximal flat factor of $Y$, i.e. $Y=Y^0\times Y^1$ where $Y^1$ is a Riemannian symmetric space with a semisimple isometry group. Assume that the Cahen--Wallach factor $\CW$ is of dimension $n+2$, $n\geq1$, and write $X=\mathcal{P}\times Y^1$ where $\mathcal{P}\mathrel{\mathop:}=\CW\times Y^0$ is then a  \textit{decomposable symmetric plane wave}. 
\begin{lemma}
    The factor $\mathcal{P}$ has a unique, up to scale, parallel lightlike vector field.
\end{lemma}
\begin{proof}Let $V$ be the unique (up to scaling) parallel vector field on $\CW$. Assume that there is another parallel lightlike vector field $V'$. Then, the plane field $E$ generated by $V,V'$ is a parallel Lorentzian plane field. In particular, $E$ integrates to a foliation  $\mathcal{E}$ of $\mathcal{P}$ by Lorentzian surfaces. The leaves of  the foliation $\mathcal{E}$ are all flat, since they admit a basis of complete parallel (lightlike) vector fields. On the other hand, the orthogonal distribution $E^\perp$ is also parallel. So, it is integrable and totally geodesic.  However, as $E^\perp\subset V^\perp$,  and the leaves tangent to  $V^\perp$ are flat, since  $\mathcal{P}$ is a (symmetric) plane wave. It follows that the leaves tangent to $E^\perp$ are also flat. Now, decomposing $T\mathcal{P}$ as a the orthogonal product $E\times E^\perp$, as both parts are flat, it follows that $\mathcal{P}$ is flat which is a contradiction.
\end{proof}

\begin{fact}[Theorem 1.5 in \cite{hanounah2025homogeneous}]\label{isometry plane wave}
    The isometry group of a complete homogeneous plane wave $\mathcal{P}^{\prime}$ (indecomposable or not) of dimension $n+2$ that admits a unique parallel lightlike vector field is  isomorphic, up to finite index, to $(\R\times K) \ltimes \Heis_{2n+1}$ where $K\subset \O(n)$ is a compact connected subgroup. Moreover, the center of $\Heis_{2n+1}$ is central in a subgroup of index two in $\Isom(\mathcal{P}^{\prime})$. In addition, the flow of the (unique) parallel lightlike vector field is given by the action of center of $\Heis_{2n+1}$ on $\mathcal{P}^{\prime}$.
\end{fact}

\begin{lemma}\label{Lemma: CW splitting}
    We have $\Isom(X)=\Isom(\mathcal{P})\times \Isom(Y^1)$ where $X=\mathcal{P}\times Y^1$ is as above. 
\end{lemma}
\begin{proof}
    We proceed as in the proof of Proposition \ref{splitting isometry}. Let $p=(p_1,p_2)\in \mathcal{P}\times Y^1$. We have $\Hol_p=\Hol_{p_1}\times \Hol_{p_2}\subset \SO(T_{p_1}\mathcal{P})\times \SO(T_{p_2}Y^1)$ and it is enough to show that $\Stab(p)$ normalizes $\Hol_{p_2}$. Indeed, since $Y$ is a Riemannian symmetric space without a flat factor then its holonomy $\Hol_{p_2}$ coincides with the isotropy \cite[Proposition 10.79]{besse2007einstein}. On the other hand, $\Hol_{p_1}$ is abelian isomorphic to $\R^{n}$ (see \cite[Proposition 3]{leistner2016completeness}). So $\Stab(p)$ normalizes $\R^n\times \Hol_{p_2}$ which implies that it also normalizes $\Hol_{p_2}$. The fixed subspace of the $\Hol_{p_2}$-action is the subspace at $p$ tangent to $\mathcal{P}$. The rest is exactly as in the proof of Proposition \ref{splitting isometry}.
\end{proof}
\begin{cor}
    Let $M$ be a compact $(\Isom(X),X)$-manifold. Then $M$ is, up to double cover, a symmetric Brinkmann spacetime. In particular, $M$ is geodesically complete.
\end{cor}
\begin{proof}
    By Fact \ref{isometry plane wave} and Lemma \ref{Lemma: CW splitting}, an index-two subgroup of $\Isom(X)$ preserves a lightlike parallel vector field. Therefore, up to taking a double cover, $M$ is a Brinkmann spacetime which is complete due to \cite{mehidi2022completeness}.
\end{proof}
\subsection{The timelike case}
If $X=\R\times Y$ where $Y$ is Riemannian symmetric without a flat factor, then $\Isom(X)=\Isom(\R, -dt^2)\times \Isom(Y)$. Therefore, the isotropy of a point is $\Z/2\Z\times K$ which implies that $X$ admits an $\Isom(X)$-invariant Riemannian metric. In particular, any compact $(\Isom(X),X)$-manifold is complete. 
\begin{remark}[Timelike Killing fields]
    The splitting of the isometry group implies that any compact $(\Isom(X),X)$-manifold possesses, up to a double cover, a parallel timelike vector field, which is, in particular, a Killing field. Moreover, it is shown in \cite{romero1995completeness} that compact Lorentz manifolds admitting timelike Killing fields are necessarily complete.

\end{remark}

\section{Low dimensional compact locally symmetric Lorentz manifolds}
A $2$-dimensional compact connected locally symmetric Lorentz manifold is,  up to double cover, isometric to a Lorentz flat torus. In the $3$-dimensional case, the classification is more involved and in fact contained in the works of many people. For instance, when $M$ has  a constant curvature $\kappa$, it follows that $\kappa\leq 0$ by \cite{calabimarkus,klingler1996completude}. In the flat case, $M$ is up to finite cover a \textit{solvmanifold}, i.e. $M = \Gamma\backslash G$ where $G$ is a connected three dimensional unimodular solvable Lie group. The latter fact follows from \cite{carriere1989autour,fried1983three} which is extended to higher dimensions by \cite{goldman1984fundamental,carriere1989autour}, showing that all compact flat Lorentz manifolds are finitely covered by solvmanifolds. In the anti de Sitter case, there are many works around the Geometry and Topology of compact $\AdS^{1,2}$-manifolds, see for example \cite{Kulkarni,kulkarniraymond,goldmannonstandard, ghysAnosov,Mess,zeghibclosed,salein,gueritaud2015compact}.\\

In the non-constant curvature case, indecomposable symmetric spaces are Cahen--Wallach spaces. In \cite{KO} a systematic study of the existence of compact manifolds that are locally isometric to Cahen--Wallach spaces is achieved. In particular, there are no compact quotients of $3$-dimensional Cahen--Wallach spaces (see also \cite{DZ}).\\

So we are left only with the decomposable cases in dimension $3$. Namely, the symmetric spaces $\widetilde{\dS}^{1,1}\times \R$, $\R\times\mS^2$ and $\R\times\mathbb{H}^2$. Only the last two examples admit compact quotients. 

\subsection{The $4$-dimensional case}
\subsubsection{The flat case} In addition to the fact that compact flat Lorentz manifolds are up to finite cover solvmanifold \cite{goldman1984fundamental}, an (essential) classification of their fundamental groups (in every dimension) is established in \cite{grunewaldMargulis}. It follows, in particular, from \cite[Theorem 1.10]{grunewaldMargulis} that the fundamental group of a compact flat Lorentz $4$-manifold is (virtually) either nilpotent or isomorphic to $\Z\ltimes \Z^3$.
\subsubsection{The Cahen--Wallach case} Concerning their topology, compact Cahen--Wallach are not always standard, in particular, their isometry groups are not always algebraic! In fact, there are non-standard examples even in dimension $4$ \cite[Example 6.3]{hanounah2025topology}, in contrast to the $3$-dimensional situation. On the other hand, their fundamental groups are fairly understood. They are either virtually nilpotent (in this case they are nilmanifolds) or virtually isomorphic $\Z\ltimes \Gamma_0$ where $\Gamma_0$ is a lattice of $\Heis_3$ (see \cite[Section 8]{KO}).
\subsubsection{Anti de Sitter and de Sitter cases}
Due to a variant of the Gauss-Bonnet formula, compact anti de Sitter spaces do not exist in even dimensions \cite[Corollary 2.10]{Kulkarni}. On the other hand, the Calabi--Markus phenomenon eliminates also the de Sitter case. 
\subsubsection*{The decomposable cases} It remains to treat the following list of decomposable Lorentz symmetric spaces:
\[\R\times \mathbb{H}^3, \ \ \R\times \mathbb{S}^3, \ \ \R^{1,1}\times \mathbb{H}^2, \ \ \R^{1,1}\times \mathbb{S}^2, \ \ \widetilde{\dS}^{1,1}\times \mathbb{R}^2, \ \ \widetilde{\dS}^{1,1}\times \mathbb{H}^2, \ \ \widetilde{\dS}^{1,1}\times \mathbb{S}^2,\] 

\[ \widetilde{\AdS}^{1,2}\times \R, \ \  \dS^{1,2}\times \R, \ \ \operatorname{CW}_h\times \R, \ \ \operatorname{CW}_e\times \R \]
\\
where $\operatorname{CW}_h$ and $\operatorname{CW}_e$ denote the (indecomposable) Cahen--Wallach spaces of \textit{hyperbolic} and \textit{elliptic} types respectively (we note that $\AdS^{1,1}$ and $\dS^{1,1}$ are the same homogeneous space, they are both equal to $\SO(1,2)/\SO(1,1)$). A ``classification'' of all compact quotients of each space is achievable. When the maximal flat Lorentz factor is non-trivial we have the following result on the ``rationality'' of the flat leaves. 

\begin{pr}\label{rationality}
    Let $X$ be either $\R^{1,1}\times \mathbb{H}^2$ or $\R\times \mathbb{H}^3$ and $M$ be a compact manifold locally modeled on $X$. Then the flat leaves of $M$ are closed.
\end{pr}

\begin{lemma}[\cite{goldman1984fundamental}]\label{fact: normalizer}
Let $G$ be a non-compact connected amenable subgroup of $\O(1,n)$. Then the normalizer of $G$ is amenable.
\end{lemma} 
\begin{proof}[Proof of Proposition \ref{rationality}]
    Let $(D,\rho)$ be a developing pair for $M$ and put $\Gamma\mathrel{\mathop:}=\rho(\pi_1(M))$.\\
    \textbf{The case where $X=\R\times \mathbb{H}^3$.} In this case $\Gamma$ is a cocompact lattice of $\R\times \SO(1,3)$. So, $\Gamma$ intersects the amenable radical of $\R\times \SO(1,3)$, which is the $\R$-factor, in a lattice, \cite[Lemma 2.1]{gelander2020minimal}.\\
    \textbf{The case when $X=\R^{1,1}\times \mathbb{H}^2$.}
Let $\Gamma\subset G\mathrel{\mathop:}=\Isom(\R^{1,1})\times \Isom(\mathbb{H}^2)$ be a discrete group that acts properly and cocompactly on $X$. The identity component $H^o$ of $H\mathrel{\mathop:}=\overline{p(\Gamma)}$ is solvable by Fact \ref{fact: auslander} where $p:G\to \Isom(\mathbb{H}^2)=\SO(1,2)$ is the projection. If $H^o$ is trivial then by Lemma \ref{useful lemma} we have that $\Gamma\cap \Isom(\R^{1,1})$ acts cocompactly on $\R^{1,1}$ and we are done. If $H^o$ is not trivial then the identity component of the normalizer $N(H^o)$ isomorphic to $\Aff^+(\R)$. Indeed, if $H^o$ is a hyperbolic (or elliptic) one-parameter group then, up to finite index, $N(H^o)$ is $H^o$ itself which is not cocompact in $\SO(1,2)$. If $H^o$ is parabolic then its normalizer is, up to finite index, a copy of $\Aff^+(\R)$. It follows that (virtually) $\Gamma\subset G'\mathrel{\mathop:}=\Sol\times \Aff(\R)$ where $\Sol=\SO(1,1)\ltimes \R^2$. This a contradiction using Proposition \ref{Prop: modular distorsion} point (2), since this would imply that $\Aff^+(\R)$  is unimodular.
\end{proof}

We conclude that if $M$ is a compact connected $4$-manifold locally modeled on $\R\times \mathbb{H}^3$ (resp. on $\R^{1,1}\times \mathbb{H}^2$), then $M$ is (geometrically) an $\mathbb{S}^1$-bundle over a closed hyperbolic $3$-manifold (resp. a $\mathbb{T}^2$-bundle over a closed hyperbolic surface).
\appendix
\section{The lightlike Poincar\'e group}\label{lightlike poincare}

Let $\langle \cdot , \cdot \rangle$ denote the Lorentz inner product on $\R^{n+2}$, with coordinates $(v,x_1,x_2,\ldots, x_n,u)$, whose quadratic form equals $2dvdu+\sum x_i^2$. Let $L_n\subset \O(1,n+1)$ be the subgroup of Lorentz linear transformations leaving invariant the lightlike hyperplane $P\mathrel{\mathop:}=\{u=0\}$. An element $B\in L_n$ has the form $$B=\begin{pmatrix}
    \lambda & \beta^{\top} & a\\
    0 & A & \alpha\\
    0 & 0 &\lambda^{-1} 
\end{pmatrix}$$
where $\lambda\in \R^*$, $a\in \R$,  $\alpha, \beta\in \R^n$, and $A\in \O(n)$. For $X\in \R^{n+2}$ we have

$$BX=\begin{pmatrix}
    \lambda & \beta^{\top} & a\\
    0 & A & \alpha\\
    0 & 0 &\lambda^{-1} 
\end{pmatrix}\begin{pmatrix}
    v\\
    \delta\\
    u 
\end{pmatrix}=\begin{pmatrix}
    \lambda v+\beta^{\top}\delta+au\\
    A\delta+u\alpha\\
    \lambda^{-1}u 
\end{pmatrix}$$
where $\delta\in \R^n$ satisfies $X=(v,\delta,u)^{\top}$. A straightforward computation shows that $$\langle BX, BX\rangle = 2vu+\left \|\delta\right \|^2+2u\left(\lambda^{-1}\beta^{\top}\delta+(A\delta)^{\top}\alpha\right)+u^2(2\lambda^{-1}a+\left \|\alpha\right \|^2)=2vu+\left \|\delta\right \|^2$$
for all $v$, $u$ and $\delta$. This shows that $$\left(\lambda^{-1}\beta^{\top}\delta+(A\delta)^{\top}\alpha\right)=\left(2\lambda^{-1}a+\left \|\alpha\right \|^2\right)=0$$
for all $\delta\in \R^n$. But, $$\left(\lambda^{-1}\beta^{\top}\delta+(A\delta)^{\top}\alpha\right)=\left(\lambda^{-1}\beta^{\top}\delta+(A^{-1}\alpha)^{\top}\delta\right)=\left(\lambda^{-1}\beta+A^{-1}\alpha\right)^{\top}\delta.$$
We obtain that $\alpha=-\lambda^{-1}A\beta$ and $a=-\frac{\lambda^{-1}}{2}\left \|\beta\right \|^2$. In other words, the elements of $L_n$ are
$$L_n=\left\{\begin{pmatrix}
    \lambda & \beta^{\top} & -\frac{\lambda^{-1}}{2}\left \|\beta\right \|^2\\
    0 & A & -\lambda^{-1}A\beta\\
    0 & 0 &\lambda^{-1} 
\end{pmatrix}, \ \lambda\in \R^*, \ \beta\in \R^n, \ A\in \O(n) \right\}.$$

One verifies that $L_n\simeq \left(\R^*\times \O(n)\right)\ltimes \R^n\simeq \Sim(\R^n)$, where $\Sim(\R^n)$ denotes the affine similarity group of $\R^n$. We denote by $L_n^+$ the index-two subgroup of $L_n$ with $\lambda>0$.\\

The \textit{lightlike Poincar\'e group} $\operatorname{LP}(n)$ is the subgroup $\operatorname{LP}(n)\mathrel{\mathop:}=L_n^+\ltimes P$ of the Poincar\'e group $\O(1,n+1)\ltimes \R^{n+2}$. We can rewrite $\operatorname{LP}(n)$ as
$$\operatorname{LP}(n)=\left(\left(\R\times \O(n)\right)\ltimes \R^n\right)\ltimes\R^{n+1}=\left(\R\times\O(n)\right)\ltimes (\R^{n}\ltimes\R^{n+1})=\left(\R\times\O(n)\right)\ltimes \Heis_{2n+1}$$
where $\left(\R\times\O(n)\right)$ acts on $\Heis_{2n+1}$ as follows: we fix a basis of the Heisenberg algebra given by $Z_0,X_1,\cdots,X_n,Y_1,\cdots, Y_n$ where the only non-vanishing brackets are $[X_i,Y_i]=Z_0$. The abelian algebra generated by $X_1,\cdots,X_n$ gives the nilradical of the  group $L_n$. The the $\R$-action is the exponential of the following derivation written in the above basis:
\begin{equation*}
    \begin{pmatrix}
        \ln(\lambda) & 0&0\\
        0&\ln(\lambda) \Id_n& 0\\
        0&0&0
    \end{pmatrix}.
\end{equation*}
On the other hand the an element of $\O(n)$-factor is given by the exponential of  the following derivation
\begin{equation*}
    \begin{pmatrix}
        0 & 0&0\\
        0&A& 0\\
        0&0&A
    \end{pmatrix},
\end{equation*}
where $A$ is an anti symmetric matrix.
\subsection{Action on half-Minkowski}\label{homogeneous half Minkowski} Let $C_n$ denote the half-Minkowski space, defined as $C_n\mathrel{\mathop:}=\{u>0\}$. The group $\operatorname{LP}(n)$ acts transitively on $C_n$. Let $I_n\subset \operatorname{LP}(n)$ be the stabilizer inside $\operatorname{LP}(n)$ of the vector $\operatorname{u}=(0,0,\ldots,0,1)\in C_n$. An element $E\in I_n$ has the form
$$E=\begin{pmatrix}
    \lambda & \beta^{\top} & -\frac{\lambda^{-1}}{2}\left \|\beta\right \|^2\\
    0 & A & -\lambda^{-1}A\beta\\
    0 & 0 &\lambda^{-1} 
\end{pmatrix}+\begin{pmatrix}
    v \\
    \delta \\
    0 
\end{pmatrix}.$$
One checks that $\lambda=1$, $v=\frac{\left \|\beta\right \|^2}{2}$, and $\delta=A\beta$. In other words, 
$$E=\begin{pmatrix}
    1 & \beta^{\top} & -\frac{\left \|\beta\right \|^2}{2}\\
    0 & A & -A\beta\\
    0 & 0 &1 
\end{pmatrix}+\begin{pmatrix}
    \frac{\left \|\beta\right \|^2}{2} \\
    A\beta \\
    0 
\end{pmatrix}.$$
Hence, $I_n=\O(n)\ltimes J_n\subset \left(\R\times\O(n)\right)\ltimes \Heis_{2n+1}$ where $J_n\cong \R^n$ is an $\O(n)$-invariant subgroup of $\Heis_{2n+1}$ that does not intersect the center of $\Heis_{2n+1}$. So $$C_n=\left(\left(\R\times\O(n)\right)\ltimes \Heis_{2n+1}\right)/(\O(n)\ltimes J_n).$$

\vspace{0.5cm}

\paragraph{\textbf{Competing interests:}} The authors have no competing interests to declare that are relevant to the content of this article. \\\textbf{Data availibility} Data sharing is not applicable to this article as no new data were created or analyzed in this study. 
\addtocontents{toc}{\protect\setcounter{tocdepth}{1}}
\bibliographystyle{alpha}
\bibliography{Bibliography.bib}

\end{document}